\documentclass[11pt, reqno]{amsart}
\usepackage{amssymb,latexsym,amsmath,amsfonts}
\usepackage{latexsym}
\usepackage{latexsym}
\usepackage{hyperref}
\usepackage[mathscr]{eucal}
\usepackage[english]{babel}
\numberwithin{equation}{section}

\newtheorem{thm}{ \bf Theorem}[section]
\newtheorem{cor}[thm]{ \bf Corollary}
\newtheorem{lem}{ \bf Lemma}[section]
\newtheorem{defn}{ \bf Definition}[section]

\newtheorem{pro}{\bf Proposition}
\newcommand{\be}{\begin{equation}}
\newcommand{\ee}{\end{equation}}
\newcommand{\Bea}{\begin{eqnarray*}}
\newcommand{\Eea}{\end{eqnarray*}}
\newcommand{\bea}{\begin{eqnarray}}
\newcommand{\eea}{\end{eqnarray}}

\numberwithin{equation}{section}

\def\de{{\delta}}
\def\dg{{\delta}_g}
\def\dd{{\delta}^D}
\def\Rt{{\check{R}}}
\def\gt{\tilde{g}}
\def\la{\Delta}
\def\a{\alpha}

\def\tR{\tilde{\mathcal{R}_p}}

\def\th{\tilde{h}}

\begin{document}
\title[On the stability of $\mathcal{R}_p$]{On the stability of the $L^p$-norm of the Riemannian curvature tensor}
\author{Soma Maity}
\address{Department of Mathematics, Indian Institute of Science,
Bangalore-12, India}
\email{somamaity@math.iisc.ernet.in}

\begin{abstract}
We consider the Riemannian functional $\mathcal{R}_p(g)=\int_M|R(g)|^pdv_g$ defined on the space of Riemannian metrics with unite volume on a closed smooth manifold $M$ where $R(g)$ and $dv_g$ denote corresponding Riemannian curvature tensor and volume form and $p\in(0,\infty)$. First we prove that the Riemannian metrics with non-zero constant sectional curvature are strictly stable for $\mathcal{R}_p$ for certain values of $p$. Then we conclude that they are strict local minimizer for $\mathcal{R}_p$ for those values of $p$. Finally generalizing this result we prove that product of space forms of same type and dimension are strict local minimizer for $\mathcal{R}_p$ for certain values of $p.$
\end{abstract}
\thanks{}
\keywords{Riemannian functional, critical point, stability, local minima}
\maketitle

\tableofcontents
\addtocontents{toc}{\protect\setcounter{tocdepth}{1}}
\setcounter{equation}{0}

\section{Introduction}
Let $M$ be a closed smooth manifold of dimension $n\geq 3$ and $\mathcal{M}$ denote the space of Riemannian metrics on $M$ endowed with the $C^{2,\a}$-topology for any $\a\in (0,1)$.  In this paper we study the following Riemannian functional,
\Bea \mathcal{R}_p(g)=\int_M|R(g)|^pdv_g
\Eea
where $R(g)$ and $dv_g$ denote corresponding Riemannian curvature and volume form. Since the functional is not scale-invariant, we restrict the functional to the subspace $\mathcal{M}_1 \subset \mathcal{M}$ consisting of metrics with unit volume. For $p<\frac{n}{2}$ it was pointed out by Gromov that $\inf_g\mathcal{R}_{p|\mathcal{M}_1}=0$. Note that for $p=\frac{n}{2}$ the functional is scale-invariant. In dimension four, the Chern-Gauss-Bonnet theorem implies that  Einstein metrics give an absolute minimum $8\pi^2\chi(M)$ for the functional $\mathcal{R}_2$, where $\chi (M)$ denote the Euler characteristic of $M$. In \cite{AM2} M. T. Anderson conjectured that if $M$ be a closed hyperbolic 3-manifold then $\inf_g \mathcal{R}_{\frac{3}{2}}$ is realized by the hyperbolic metric. In this paper we study the local minimizing property of $\mathcal{R}_p$ for $p\geq 2$ at some certain critical metrics.

Before stating our results we recall a canonical decompositions $\mathcal{M}$. From \cite{BA} Lemma 4.57, if $M$ is a compact Riemannian manifold, we have the orthogonal decomposition of the tangent space of $\mathcal{M}$ at $g$(which is the space $S^2(T^*M)$ of symmetric 2-tensors on $M$):
\be T_g\mathcal{M}= S^2(T^*M) =({\rm Im}\dg^* + C^{\infty}(M).g )\oplus(\dg^{-1}(0)\cap {\rm Tr}_g^{-1}(0))
\ee
Here Im$\dg^*$ is precisely the tangent space of the orbit of $g$ under the action of the group of diffeomorphisms of $M$. Since  $ T_g\mathcal{M}_1=\{h\in S^2(T^*M)|\int_M tr(h)dv_g=0\}$, we have a corresponding decomposition
\be T_g\mathcal{M}_1=({\rm Im}\dg^* + C^{\infty}(M).g )\cap T_g\mathcal{M}_1 \oplus(\dg^{-1}(0)\cap {\rm Tr}_g^{-1}(0))\ee
$\mathcal{M}$ is an open convex subset of $S^2(T^*M)$ equipped with $C^{2,\a}$-topology. Since $S^2(T^*M)$ is a vector space we can differentiate  $\mathcal{R}_p$ on $\mathcal{M}$ along any vector in $S^2(T^*M)$. $\nabla \mathcal{R}_p(g)$ in $S^2(T^*M)$ is called the {\it gradient} of $\mathcal{R}_p$ at $g$ if for every $h\in S^2(T^*M)$,
$$\frac{d}{dt}_{|t=0}\mathcal{R}_p(g+th)= \mathcal{R}_{p|g}'.h=\langle \nabla \mathcal{R}_p(g),h\rangle$$
$g$ is called a {\it critical point} for $\mathcal{R}_{p}$ if the component of $\nabla\mathcal{R}_p(g)$ along $T_g\mathcal{M}_1$ is zero. By a standard technique one can prove that every compact irreducible locally symmetric space is a critical point of $\mathcal{R}_p.$ Let $g$ be a critical point of $\mathcal{R}_{p}$. The {\it Hessian} $H$ of $\mathcal{R}_p$ is a symmetric bilinear map,
$$H:T_g\mathcal{M}_1\times T_g\mathcal{M}_1\to \mathbb{R}$$
defined by
$$H(h_1,h_2)=\frac{\partial}{\partial t}\frac{\partial}{\partial s}\mathcal{R}_p(g(s,t))_{|t=0,s=0}$$
where $g(s,t)$ is a two-parameter family of metrics in $\mathcal{M}_1$ with
$g(0,0)=g$ and $\frac{\partial}{\partial t}g(t,0)_{|t=0}=h_1$, $\frac{\partial}{\partial s}g(0,s)_{|s=0}=h_2$.
\vspace{2mm}

Let $\mathcal{W}$ denote the orthogonal complement of ${\rm Im}\dg^*$ in $T_g\mathcal{M}_1$.
\begin{defn}{\rm Let $(M,g)$ be a critical point for $\mathcal{R}_{p|\mathcal{M}_1}$. The metric $g$ is called {\it infinitesimally rigid} for $\mathcal{R}_p$ if the kernel of the bi-linear form $H$ restricted to $\mathcal{W}\times \mathcal{W}$ is zero}.
\end{defn}
In \cite{MY}, Y. Muto proved that ($S^n$, can) is infinitesimally rigid for ${\mathcal{R}}_2$. For $p=2$, the application of the differential Bianchi identity simplifies the expression for the gradient of ${\mathcal{R}}_2$. So it is easier to study the second variation of ${\mathcal{R}}_2$ than $\mathcal{R}_p$ for any arbitrary $p$, at a critical point. However it is not known that ${\mathcal{R}}_2$ is infinitesimally rigid even for any arbitrary irreducible symmetric space.
\begin{defn} {\rm  Let $(M ,g)$ be a critical point for $\mathcal{R}_p$. $(M,g)$ is {\it strictly stable} for $\mathcal{R}_p$ if there is an $\epsilon>0$ such that for every element $h$ in $\mathcal{W}$,
\be H(h,h)\geq \epsilon \|h\|^2
\ee
where $\|.\|$ denote the $L^2$-norm on $S^2(T^*M)$ defined by $g$}.
\end{defn}
\vspace{3mm}
For a metric with constant sectional curvature or product of metrics with constant sectional curvature we prove that $\mathcal{R}_p$ is infinitesimally rigid. In fact we prove that $\mathcal{R}_p$ is {\it strictly stable} for these metrics.
\begin{thm}\label{main}
Let $(M,g)$ be a closed Riemannian manifold with dimension $n\geq 3$. If $(M,g)$ is one of the following then $g$ is strictly stable for $\mathcal{R}_p$ for the indicated values of $p$:

(i) A spherical space form and $p\in[2,\infty)$.

(ii) A hyperbolic manifold and $p\in[\frac{n}{2},\infty)$.

(iii) A product of spherical space forms and  $p\in[2,n]$.

(iv) A product of hyperbolic manifolds and $p\in[\frac{n}{2},n]$.\\
Moreover, in all these cases, $H$ is diagonalizable with respect to the decomposition (1.2), for all $p \in[2,\infty)$.
\end{thm}
The product of a spherical space form and a compact hyperbolic manifold with the same dimension is a critical point of $\mathcal{R}_p$ but we are not able to prove that this is stable for $\mathcal{R}_p$. From the proof of the theorem we observe the following Proposition, which gives some information in the hyperbolic case when $ p \le \frac{n}{2}$.
\begin{pro}
Let $(M,g)$ be a compact hyperbolic manifold with the sectional curvature $c$. If the first positive eigenvalue of the Laplacian $\lambda_1$ satisfies the inequality $$\lambda_1>\frac{|c|(n-2p)}{n+2p+4}$$
then $g$ is strictly stable for $p\in [2, \frac{n}{2})$.
\end{pro}
\begin{defn} {\rm Let $(M,g)$ be a critical metric for $\mathcal{R}_{p|\mathcal{M}_1}$. Then $g$ is called a {\it strict local minimizer} if there exists a $C^{2,\alpha}$-neighborhood $\mathcal{U}$ of $g$ in $\mathcal{M}_1$, such that for all metrics $\tilde{g}\in \mathcal{U}$,
\Bea \mathcal{R}_p(\tilde{g})\geq \mathcal{R}_p(g)
\Eea
The equality holds if and only if $\tilde{g}=\phi^*g$ for some $C^{3,\alpha}$-diffeomorphism $\phi:M\rightarrow M$.}
\end{defn}
Since $\mathcal{M}$ and its sub-manifolds are Fr\'{e}chet manifolds modeled on $S^2(T^*M)$, the usual inverse function theorem can not be applied. Using the Slicing Lemma 2.10 in [GV], we observe that {\it if $(M,g)$ is a closed Riemannian manifold such that $g$ is strictly stable then it is a strict local minimizer for $\mathcal{R}_p$.}

Similar results have been  proved by Besson, Courtois and Gallot in [BCG2] for all irreducible locally symmetric spaces of non-compact type for the functional
$$\int_M |s|^{\frac{n}{2}}dv_g$$
where $s$ denote the scalar curvature of $g$. \vspace{2mm}

In section 4, we study the second variation of $\mathcal{R}_p$ at metrics with constant curvature and prove (i) and (ii) part of the theorem using the decomposition (1.2). We first prove that for any $h\in (\dg^{-1}(0)\cap {\rm Tr}_g^{-1}(0))$, there exists an $\epsilon_0>0$ such that $H(h,h)\geq\epsilon_0 \|h\|^2$ for all $p\geq 2$ in this case. We use a Bochner type formula to prove this step.

Next, we study the second variation of $\mathcal{R}_p$ along the conformal variations of the metric.  A positive lower bound of the Ricci curvature gives a lower bound for the first eigenvalue of the Laplacian for compact manifolds. Using this estimate we prove that for any $f\in C^{\infty}(M)$, there exists an $\epsilon_1>0$ such that
\be H(fg,fg)\geq \epsilon_1\|fg\|^2\ee
for metrics with constant positive sectional curvature for $p\geq 2$. When the sectional curvature is negative (1.4) follows immediately for $p\geq \frac{n}{2}$ from the expression of $H(h,h)$ we obtain in this section. For $p<\frac{n}{2}$, if the first eigenvalue of the Laplacian $\lambda_1$ satisfies the inequality $\lambda_1>\frac{|c|(n-2p)}{n+2p+4}$, ($c$ is the sectional curvature), then $H$ satisfies (1.4).

Finally, proving that $H$ is diagonalizable by the decomposition (1.2) for all $p\geq 2$, we get the desired result.

In section 5, we prove (iii) and (iv) part of the theorem. The main steps of the proof are similar to the proof of (i) and (ii). In section 6, we study the local minimization property of $\mathcal{R}_p$.
\subsection{Acknowledgments:} I would like to thank Harish Seshadri for suggesting this problem and his guidance, Atreyee Bhattacharya and Gururaja H. A. for some useful discussions related to this article. This work is supported by CSIR and partially supported by UGC Center for Advanced Studies.
\\
\\
\section{Index of notations and definitions} The following notations and definitions will be used throughout this paper. Let $(M,g)$ be a Riemannian manifold with dimension $n\geq 3$.\\
\\
$R,r,s$: $(4,0)$ Riemannian curvature tensor, Ricci curvature, Scalar curvature respectively\\
\\
$dv_g$, $V(g):$ The volume form and the volume of $(M,g)$\\
\\
(  ,  ), $|\ . \ |$ : The point-wise inner product and norm in the fibers of a various tensor bundle $M$ defined by $g$\\
\\
$\langle \ , \ \rangle$, $\| . \|$: The global inner-product and norm defined on the space of sections of a tensor bundle on $M$ induced by $g$ \\
\\
$D, D^*$: The Riemannian connection and its formal adjoint.\\
\\
$S^2(T^*M)$:  The sections of symmetric 2-tensor bundle over $M$\\
\\
$d^D:S^2(T^*M)\to \Gamma (T^*M\otimes\Lambda^2M)$ defined by $d^D\alpha(x,y,z):= (D_y\alpha)(x,z)-(D_z\alpha)(x,y)$ where $\Lambda^2M$ the space of denotes alternating 2-tensors and $\Gamma (T^*M\otimes\Lambda^2M)$ denotes the sections of $(T^*M\otimes\Lambda^2M)$.\\
\\
Its formal adjoint $\dd$ is defined by, $\dd(A)(x,y)=\sum\{D_{e_i}A(x,y,e_i) +D_{e_i}A(y,x,e_i)\}$\\
\\
where $\{e_i\}$ is an orthonormal basis at a point $x\in M$. \\
\\
$\Rt(x,y):=\sum R(x,e_i,e_j,e_k)R(y,e_i,e_j,e_k) $ \\
\\
Next, consider a one-parameter family of metrics $g(t)$ with $g(0)=g$ and
$h:=\frac{\partial}{\partial t}g(t)_{|t=0}$.\\
Define,
$\Pi_h(x,y)= \frac{\partial}{\partial t}D_xy_{|t=0}$ and $C_h(x,y,z):=(Pi_h(x,y),z)$.
A simple calculation shows that $C_h(x,y,z)=\frac{1}{2}[D_xh(y,z)+D_yh(x,z)-D_zh(x,y)]$
where $x$, $y$, $z$ are fixed vector fields on $M$. The suffix $h$ will be omitted when there is no ambiguity.\\
\\
$\bar{R}_h:= \frac{\partial }{ \partial t}R_{|t=0} $  and $\bar{r}_h(x,y):=\bar{R}_h(x,e_i,y,e_i)$\\
\\
$\dg:S^2(T^*M)\to \Omega^1(M)$ defined by $\dg(h)(x)=-D_{e_i}h(e_i,x)$\\
\\
Its formal adjoint $\dg^*$ defined by $\dg^*\omega(x,y):= \frac{1}{2}(D_xy+D_yx)$.\\
\\
$L$: A $(0,3)$-tensor defined by,
\Bea L_h(w,y,z):&=&\sum[R(y,z,\Pi(e_i,e_i),w)+R(y,z,e_i,\Pi(e_i,w))+R(z,e_i,\Pi(y,e_i),w)\\
&&+R(z,e_i,e_i,\Pi(y,w))+R(e_i,y,\Pi(z,e_i),w)+R(e_i,y,e_i,\Pi(z,w))]
\Eea
$W_h:=(D^*)'(h)(R)-L_h$\\
\\
$d$, $\delta$ : The exterior derivative acting on the space of deferential forms and its formal adjoint.\\
\\
$\la$: The Laplace operator acting on $C^{\infty}(M)$ defined by $\la f=\delta df= -trDdf$.
\section{Gradient of $\mathcal{R}_p$}
In this section, we compute the Euler-Lagrange equation of $\mathcal{R}_p$.

\begin{pro}The functional $\mathcal{R}_p$ is differentiable with the gradient
$$  \nabla\mathcal {R}_{p|\mathcal{M}}=-p\dd D^*|R|^{p-2}R-p|R|^{p-2}\Rt+\frac{1}{2}|R|^pg $$
and
\Bea \nabla\mathcal {R}_{p|\mathcal{M}_1}=-p\dd D^*|R|^{p-2}R-p|R|^{p-2}\Rt+\frac{1}{2}|R|^pg +(\frac{p}{n}-\frac{1}{2})\|R\|^pg
 \Eea
\end{pro}

\begin{proof}
\Bea(\mathcal{R}'_p)_g(h)=
\int_M \frac{\partial}{ \partial t}|R|^pdv_{g| t=0}
+\frac{1}{2} \int_M|R|^ptr(h) dv_g\Eea
\Bea   (|R|^p)_g'(h) = \frac{\partial}{\partial t} (|R|^2)^{\frac{p}{2}}_{|t=0}
= p|R|^{p-2}(R,R'_g.h)-2p|R|^{p-2}(\Rt,h)
\Eea
From Proposition 4.70 in \cite{BA} we have
\Bea R'_g.h(x,y,z,t)=D_yC(h)(x,z,t)-D_xC(h)(y,z,t)+R(x,y,z,h^{\sharp}(t)).
\Eea
Since $R$ is skew-symmetric in 1st and 2nd entries,
$$ (|R|^{p-2}R,R'_g(h))=-2(|R|^{p-2}R,DC(h))+(|R|^{p-2}\Rt,h).$$
Therefore,
 \Bea\langle|R|^{p-2}R,R'_g(h)\rangle
&=& -2\langle|R|^{p-2}R,DC(h)\rangle+\langle|R|^{p-2}\Rt,h\rangle\\
&=& -2\langle D^*|R|^{p-2}R,C(h)\rangle+\langle|R|^{p-2}\Rt,h\rangle
 \Eea
The skew-symmetry of $D^*(|R|^{p-2}R)$ in last two entries yields,
\Bea 2\langle D^*(|R|^{p-2}R),C(h)\rangle=\langle D^*(|R|^{p-2}R),d^D(h)\rangle. \Eea
This implies,
 \Bea\langle|R|^{p-2}R,R'_g.h\rangle=-\langle \dd D^*|R|^{p-2}R,h\rangle + \langle|R|^{p-2}\Rt,h\rangle .\Eea
Hence,
\Bea \mathcal {R}'_g.h= -p\langle\dd D^*|R|^{p-2}R,h\rangle-p\langle|R|^{p-2}\Rt , h\rangle+ \frac{1}{2}\langle|R|^pg,h\rangle
\Eea
Therefore,
\Bea \nabla\mathcal {R}_{p|\mathcal{M}}=-p\dd D^*|R|^{p-2}R-p|R|^{p-2}\Rt+\frac{1}{2}|R|^pg
\Eea
Now, $$\int_M tr(\nabla\mathcal{R}_p)dv_g=(\frac{n}{2}-p)\|R\|^p$$
Therefore,
\be \nabla\mathcal {R}_{p|\mathcal{M}_1}=-p\dd D^*|R|^{p-2}R-p|R|^{p-2}\Rt+\frac{1}{2}|R|^pg +(\frac{p}{n}-\frac{1}{2})\|R\|^pg
\ee
\end{proof}
By a standard technique one can easily check that every compact isotropy irreducible homogeneous space, and in particular every irreducible symmetric space is a critical point for $\mathcal{R}_p$. Let $(M_1,g_1)$ and $(M_2,g_2)$ be two homogeneous critical points of $\mathcal{R}_p$ with $|R|_{g_1}=|R|_{g_2}\neq 0$. Then $(M_1\times M_2, g_1+g_2)$ is a critical metric for $\mathcal{R}_p$ if and only if there dimensions are the same.
\section{Second Variation at space forms} In this section, we study second variation of $\mathcal{R}_{p}$. Let $(M,g)$ be a closed locally symmetric space and $h_1,h_2\in S^2(T^*M)$. Then
\Bea H(h_1,h_2)&=&\langle (\nabla \mathcal{R}_{p|\mathcal{M}_1})'_g(h_1),h_2\rangle\\
&=& -p\langle (\dd D^*(|R|^{p-2}R))'_g(h_1),h_2\rangle
-p\langle(|R|^{p-2})'_g(h_1)\Rt,h_2\rangle
-p\langle|R|^{p-2}(\Rt)'_g(h_1),h_2\rangle\\
&&+\frac{1}{2}\langle(|R|^p)'_g(h_1)g,h_2\rangle + \frac{1}{2}|R|^p\langle h_1,h_2\rangle+(\frac{p}{n}-\frac{1}{2})\|R\|^p\langle h_1,h_2\rangle
\Eea
Since $g$ is homogeneous and $R$ is parallel,
\Bea (\dd D^*(|R|^{p-2}R))'_g(h_1)&=&(\dd)'_g(h_1)D^*(|R|^{p-2}R)+
\dd (D^*)'_g(h_1)(|R|^{p-2}R)\\
&&+\dd D^*((|R|^{p-2})'_g(h_1)R)
+\dd D^*(|R|^{p-2}R'_g(h_1))\\
&=&|R|^{p-2}(D^*)'_g(h_1)R
+|R|^{p-2}\dd D^*\bar{R}_{h_1}
+\dd D^*((|R|^{p-2})'_g(h_1)R)
\Eea
Since $g$ satisfies the equation (3.1), $\Rt =\frac{1}{n}|R|^2g$.
Hence,
\bea H(h_1, h_2) &=& -p|R|^{p-2}(\langle \dd (D^*)'_g(h_1)R,h_2\rangle+\langle D^*\bar{R}_{h_1},d^D h_2\rangle)
-p|R|^{p-2}\langle\Rt '_g(h_1),h_2\rangle\\
\nonumber&&-p\langle(|R|^{p-2})'_g(h_1)R, Dd^D h_2\rangle-\frac{p}{n}|R|^2\langle(|R|^{p-2})'_g(h_1)g,h_2\rangle\\
\nonumber&&+\frac{1}{2}\langle(|R|^{p})'_g(h_1)g,h_2\rangle
+\frac{p}{n}\|R\|^p\langle h_1,h_2\rangle
\eea
Next, we assume $(M,g)$ to be a Riemannian manifold with non-zero constant sectional curvature throughout this section. We need following lemma to prove (i) and (ii) part of the theorem.
\begin{lem}
Let $(M,g)$ be a Riemannian manifold with non-zero constant sectional curvature $c$. Then\\

(i) $(\Rt)'_g.h=2c^2(n+1)h-4c^2tr(h)g+2c[-2\dg^*\dg h-Ddtr(h) +D^*Dh]$
\\

(ii) $\dd W_h = c(n-2)\dd d^Dh+ 2cDdtr(h) +2c\la tr(h)g $
\\

(iii) $D^*\bar{R}_h=-d^D\bar{r}_h-L_h $
\\

(iv) $\bar{r}_h=\frac{1}{2}[2(n-1)ch-2\dg^*\dg h-Ddtr(h)+D^*Dh]$
\\

(v) $\dd d^Dh=2D^*Dh-2\dg^*\dg h +2nch-2ctr(h)g$
\\

(vi) $(|R|^p)'_g.h=-2pc|R|^{p-2}\big(2tr\dg^*\dg h-\la tr(h)+(n-1)ctr(h)\big)$.
\\
\end{lem}
\subsection{Proof of Lemma 4.1:}
Let $\gt(t)$ be a one-parameter family of Riemannian metrics with $\gt(0)=g$ and $\gt'(0)=h$. Choose a normal coordinate $\{e_i\}$ with respect to $g$. Let $D$ be the Riemannian connection corresponding to $g$.\\
\\
{\bf Proof of (i) and (iv):}
\Bea \Rt _{pq}=\gt^{i_1i_2}\gt^{j_1j_2}\gt^{k_1k_2}R_{pi_1j_1k_1}R_{qi_2j_2k_2}
\Eea Therefore,
\Bea (\Rt_g.h)'_{pq}&=&(\gt^{i_1i_2})'\gt^{j_1j_2}\gt^{k_1k_2}R_{pi_1j_1k_1}R_{qi_2j_2k_2}
          +\gt^{i_1i_2}(\gt^{j_1j_2})'\gt^{k_1k_2}R_{pi_1j_1k_1}R_{qi_2j_2k_2}\\
          &&+\gt^{i_1i_2}\gt^{j_1j_2}(\gt^{k_1k_2})'R_{pi_1j_1k_1}R_{qi_2j_2k_2}
+\gt^{i_1i_2}\gt^{j_1j_2}\gt^{k_1k_2}(R_{pi_1j_1k_1})'R_{qi_2j_2k_2}\\
&&+\gt^{i_1i_2}\gt^{j_1j_2}\gt^{k_1k_2}R_{pi_1j_1k_1}(R_{qi_2j_2k_2})'
\Eea
 Note that  $(\gt^{ij})'=-\gt^{im}h_{mn}\gt^{nj}$.\\
Therefore,
\Bea (\Rt_g.h)'_{pq}&=&-h_{mn}\left(R_{pmij}R_{qnij}+R_{pimj}R_{qinj}+R_{pijm}R_{qijn}\right)\\
&&+(R'_g.h)_{pijk}R_{qijk}+R_{pijk}(R'_g.h)_{qijk}
\Eea
Since $R(0)=cI$, $R_{ijij}=-R_{ijji}=c$, for all $1\leq i$, $j\leq n$, otherwise $R_{ijkl}=0$.\\
This implies,
\Bea &&\sum_{m,n,i,j} [h_{mn}(R_{pmij}R_{qnij}+R_{pimj}R_{qinj}+R_{pijm}R_{qijn})]=2(n-3)c^2h_{pq}+4c^2tr(h)g_{pq}
\Eea
and
\Bea (R'_g(h))_{pijk}R_{qijk}=(R'_g(h))_{piqi}R_{qiqi}+(R'_g(h))_{piiq}R_{qiiq}
=2c(R'_g(h))_{piqi}
\Eea
and
$$(R'_g(h))_{qijk}R_{pijk}=2c(R'_g(h))_{qipi}=2c(R'_g(h))_{piqi}$$
From [1.174(c)] in \cite{BA}, we have,
\Bea 2(R'_g(h))_{piqi}&=&[(D^2_{iq}h)_{pi}+(D^2_{pi}h)_{qi}-(D^2_{pq}h)_{ii}-(D^2_{ii}h)_{pq}+h_{ij}R_{piqj} -h_{qj}R_{piij}]
\Eea
Using the Ricci identity we have,
\Bea\Sigma_i[(D^2_{iq}h)_{pi}+(D^2_{pi}h)_{qi}]
&=&\Sigma_i[(D^2_{iq})h_{pi}-(D^2_{qi}h)_{pi}+(D^2_{qi}h)_{pi}+(D^2_{pi}h)_{qi}]\\
&=&\Sigma_{i,j}[h_{ij}R_{iqpj}+h_{pj}R_{iqij}]-D\dg h_{pq}-D\dg h_{qp}\\
&=&\Sigma_{i,j}[h_{ij}R_{iqpj}+h_{pj}R_{iqij}]-2\dg ^*\dg h_{pq}
\Eea
Therefore,
\Bea 2(R'_g(h))_{piqi}&=&h_{ij}R_{iqpj}+h_{pj}R_{iqij}-2\dg ^*\dg h_{pq}-Ddtr(h)_{pq}+D^*Dh_{pq}+h_{ij}R_{piqj}-h_{qj}R_{piij}
\Eea
Using $R=cI$ again we obtain,
$$h_{ij}R_{iqpj}+h_{pj}R_{iqij} +h_{ij}R_{piqj}-h_{qj}R_{piij}=2(n-1)ch_{pq}$$
Combining these two equations, the proof of Lemma 4.1(iv) follows.\\
Next,
\Bea (\Rt'_g(h))_{pq}&=&-2(n-3)c^2h_{pq}-4c^2tr(h)g_{pq}+4c\Sigma_{i,j}(R'_g.h)_{piqi}\\
&=& 2(n+1)c^2h_{pq}-4c^2tr(h)g_{pq}+2c[-2\dg ^*\dg h_{pq}-Ddtr(h)_{pq}+D^*Dh_{pq}]
\Eea
This completes the proof of Lemma 4.1 (i).
\hfill
$\square$\\
\\
{\bf Proof of (ii):} Let $T$ be a $(0,4)$ tensor independent of $t$. Then using the expression for $D^*$ in a local coordinate chart and differentiating it with respect to $t$ we obtain,
\Bea (D^*)'_g(h)(T)(x,y,z)=-(\gt^{kj})'(D_kT)_{jxyz}+\gt^{kj}[T_{\Pi_{kj}xyz}
+T_{j\Pi_{kx}yz}+T_{jx\Pi_{ky}z}+T_{jxy\Pi_{kz}}]
\Eea
Note that, $\Pi$ is a vector valued symmetric two form. Next,
\Bea (D^*)'_g(h)(R)_{jkl}= R_{\Pi_{ii}jkl}+R_{i\Pi_{ij}kl}+R_{ij\Pi_{ik}l}+R_{ijk\Pi_{il}}.
\Eea
By the definition of $L_h$,
\Bea L_{hjkl}=\{R_{kl\Pi_{ii}j}+R_{kli\Pi_{ij}}+R_{li\Pi_{ik}j}
+R_{ik\Pi_{il}j}+R_{lii\Pi_{kj}}+R_{iki\Pi_{lj}}\}
\Eea
Combining these two and using the symmetries of $R$ we have,\\
\Bea W_{hjkl}=[R_{ij\Pi_{ik}l}+R_{ijk\Pi_{il}}-R_{li\Pi_{ik}j}
-R_{ik\Pi_{il}j}-R_{lii\Pi_{kj}}-R_{iki\Pi_{lj}}]
\Eea
Pairing it with $d^D\a$ for any $\a \in S^2(T^*M) $ and using the symmetries of
$R$ and $d^D\a$ we have,
\Bea \sum W_{hjkl}d^D\a_{jkl}
 =2\sum \big(R_{ij\Pi_{ik}l}-R_{li\Pi_{ik}j}
-R_{lii\Pi_{kj}}\big)(d^D\a)_{jkl}
\Eea
 $R=cI$ gives,
\Bea \sum R_{ij\Pi_{ki}l}d^D\a_{jkl}&=&c\sum C_{kim}R_{ijml}d^D\a_{jkl}\\
&=&c\sum C_{kii}d^D\a_{jkj}-c\sum C_{klj}d^D\a_{jkl}
\Eea
\Bea \sum R_{li\Pi_{ik}j}d^D\a_{jkl}
&=&c\sum C_{ikm}R_{limj}d^D\a_{jkl}\\
&=&c\sum C_{jkl}d^D\a_{jkl}-c\sum C_{iki}d^D\a_{lkl}
\Eea
and
\Bea \sum R_{lii\Pi_{kj}}d^D\a_{jkl}
&=&c \sum C_{kjm}R_{liim}d^D\a_{jkl}\\
&=&-(n-1)c\sum C_{jkl}d^D\a_{jkl}
\Eea
Since $C$ is symmetric in 1st two entries and $d^D\a$ is skew-symmetric in last two entries,
$$\sum C_{klj}d^D\a_{jkl}=0$$
Next a simple calculation gives, $\sum_i C_{kii}=\frac{1}{2}dtr(h)_k$
and $\sum_j d^D\a_{jkj}=dtr\a_k+\dg \a_k$.
$$\sum C_{jkl}d^D\a_{jkl}=\frac{1}{2}\sum(C_{jkl}-C_{jlk})d^D\a_{jkl}
=\frac{1}{2}\sum d^Dh_{jkl}d^D\a_{jkl}$$
Combining all these equations we have,
$$\dd W_h= (n-2)c\dd d^Dh+2cDdtr(h) +2c\la tr(h)g$$
\hfill
$\square$
\\
{\bf Proof of (iii):} Let $x$, $y$, $z$, $u$, $w$ be fixed vector fields. Then
\Bea (D_xR)'(y,z,u,w)&=&(x.R(y,z,u,w))'-\{\bar{R}_h(D_xy,z,u,w)
+\bar{R}_h(y,D_xz,u,w)\\
&&+\bar{R}_h(y,z,D_xu,w)+\bar{R}_h(y,z,u,D_xw)+R(\Pi(x,y),z,u,w)\\
&&+R(y,\Pi(x,z),u,w)+R(y,z,\Pi(x,u),w)+R(y,z,u,\Pi(x,w))\}\\
&=&D_x\bar{R}_h(y,z,u,w)-\{R(\Pi(x,y),z,u,w)+R(y,\Pi(x,z),u,w)\\
&&+R(y,z,\Pi(x,u),w)+R(y,z,u,\Pi(x,w)\}
\Eea
Applying the differential Bianchi identity we get,
\Bea(D_xR)'(y,z,u,w)+(D_yR)'(z,x,u,w)+(D_zR)'(x,y,u,w)=0
\Eea
This gives,
\Bea &&D_x\bar{R}_h(y,z,u,w)+D_y\bar{R}_h(z,x,u,w)+D_z\bar{R}_h(x,y,u,w)\\
&=&R(\Pi(x,y),z,u,w)+R(y,\Pi(x,z),u,w)+R(y,z,\Pi(x,u),w)\\
&&+R(y,z,u,\Pi(x,w))+R(\Pi(y,z),x,u,w)+R(z,\Pi(y,x),u,w)\\
&&+R(z,x,\Pi(y,u),w)+R(z,x,u,\Pi(y,w))+R(\Pi(z,x),y,u,w)\\
&&+R(x,\Pi(z,y),u,w)+R(x,y,\Pi(z,u),w)+R(x,y,u,\Pi(z,w))\\
&=&R(y,z,\Pi(x,u),w)+R(y,z,u,\Pi(x,w))+R(z,x,\Pi(y,u),w)\\
&&+R(z,x,u,\Pi(y,w))+R(x,y,\Pi(z,u),w)+R(x,y,u,\Pi(z,w))
\Eea
Consequently,
\Bea \sum(D_{e_i}\bar{R}_h)(e_i,w,y,z)&=&\sum(D_{e_i})\bar{R}_h(y,z,e_i,w)\\
&=&-\sum\{(D_y\bar{R}_h)(z,e_i,e_i,w)+(D_z\bar{R}_h)(e_i,y,e_i,w)\}+L_h(w,y,z)\\
&=&\sum\{(D_y\bar{R}_h)(z,e_i,w,e_i)-(D_z\bar{R}_h)(e_i,y,e_i,w)\}+L_h(w,y,z)\\
&=&d^D\bar{r}_h(w,y,z)+L_h(w,y,z)
\Eea
Therefore,$$D^*\bar{R}_h=-d^D\bar{r}_h-L_h.$$
\hfill
$\square$
\\
{\bf Proof of (v):}
From the identity (2.8) in \cite{BM}, we have,
\be \dd d^D h_{pq}=2D^*Dh_{pq}-2\dg^*\dg h_{pq}+\sum_i(r_{pi}h_{iq}+r_{qi}h_{ip})-2\sum_{i,j}R_{piqj}h_{ij}
\ee
A straightforward computation using $R=cI$ gives the required result.
\hfill
$\square$
\\
\\
{\bf Proof of (vi):} From the proof of Proposition 2,
\Bea (|R|^p)'_g.h&=&p|R|^{p-2}( R,R'_g.h)-2p|R|^{p-2}(\Rt,h)\\
&=&2cp|R|^{p-2}\sum (R'_g.h)_{ijij}-2\frac{p}{n}|R|^ptr(h)
\Eea
Using (iv) we have,
\Bea \sum (R'_g.h)_{ijij}&=&tr(\bar{r}_h)\\
&=&c(n-1)tr(h)-tr\dg^*\dg h +\frac{1}{2}(trD^*Dh-trDdtr(h))\\
&=& c(n-1)tr(h)-tr\dg^*\dg h+\la tr(h)
\Eea
Since $|R|^2=2c^2n(n-1)$ we have,
 $$(|R|^p)'_g(h)=-2cp|R|^{p-2}(tr\dg^*\dg h-\la tr(h)+(n-1)ctr(h))$$
\hfill
$\square$\\
Next, we study the stability of $\mathcal{R}_p$ a space forms. A symmetric covariant 2-tensor $h$ is called Transverse-Traceless tensor (TT-tensor) if $\dg h=0$ and $tr(h)=0$. First we study $H$ on TT-variations.
\subsection{Transverse-traceless Variations:} Let $(M,g)$ be a Riemannian manifold with constant sectional curvature $c\neq 0$ and $h\in \dg^{-1}(0)\cap \rm{Tr}^{-1}(0)$. In this case the expression for $ H(h,h)$ reduces to,
\Bea H(h,h)=-p|R|^{p-2}[\langle \dd (D^*)'_g.h(R),h\rangle+\langle D^*\bar{R}_h,d^Dh\rangle
+\langle \Rt'_g(h),h\rangle]+\frac{p}{n}\|R\|^p\langle h,h\rangle
\Eea
Using Lemma 4.1 (iii) we have,
\Bea H(h,h)=-p|R|^{p-2}[\langle \dd W_h,h\rangle-\langle \bar{r}_h,\dd d^Dh\rangle
+\langle \Rt'_g(h),h\rangle]+\frac{p}{n}\|R\|^p\langle h,h\rangle
\Eea
Then from the Lemma 4.1(i) we have,
\Bea \frac{p}{n}\|R\|^p\langle h,h\rangle-p\|R\|^{p-2}\langle(\Rt)'_g.h,h\rangle
&=& 2pc^2(n-1)\|R\|^{p-2}\|h\|^2\\
&& -p\|R\|^{p-2}\{(n-1)c^2\langle h, h\rangle+2c\langle Dh, Dh\rangle \}\\
&=&-2pc\|R\|^{p-2}\|Dh\|^2
\Eea
Using Lemma 4.1 (ii) and (v) we have,
\Bea\langle \dd W_h,h\rangle&=&c(n-2)\langle \dd d^Dh,h\rangle\\
&=&2c(n-2)\langle D^*Dh,h\rangle+2c^2n(n-2)\langle h,h\rangle\\
&=&2c(n-1)\|Dh\|^2+2c^2n(n-2)\|h\|^2
\Eea
Next using Lemma 4.1 (iv) and (v) we have,
\Bea \langle \bar{r}_h,\dd d^Dh\rangle
&=&-\langle 2(n-1)ch+D^*Dh,D^*Dh+nch\rangle\\
&=& -[\|D^*Dh\|^2+(3n-2)c\|Dh\|^2+2c^2n(n-1)\|h\|^2]
\Eea
Combining all these results we have,
\Bea H(h,h)&=&p\|R\|^{p-2}\{\|D^*Dh\|^2+nc\|Dh\|^2+2nc^2\|h\|^2\}
\Eea
It is clear from the above expression that if $c>0$, then $H(h,h)>2nc^2\|h\|^2$. Suppose $c<0$. Since $\|d^Dh\|^2\geq 0$ using Lemma 4.1 (v) we have that the least eigenvalue of the rough Laplacian is bounded below by $-nc$. Now,
\Bea \|D^*Dh\|^2+nc\|Dh\|^2&=& \|D^*Dh+nch\|^2-nc\langle D^*Dh+nch,h\rangle\\
&\geq& -nc\langle D^*Dh,h\rangle-n^2c^2\|h\|^2
\Eea
Hence, $H(h,h)>2nc^2\|h\|^2$.
\hfill
$\square$\\
\subsection{Conformal variations:} Next we study $H$ on the space of conformal variations of $g$. Consider any $f$ in $C^{\infty}(M)$ with $\int fdv_g=0$. In this section we prove that there exists $\epsilon_1>0$ such that
$$H(fg,fg)\geq \epsilon_1\|fg\|^2=n\epsilon_1\|f\|^2$$
First we compute each term appearing in the expression of $H$ in (4.1).
\be\frac{p}{n}\|R\|^p\| fg\|^2=2n(n-1)pc^2\|R\|^{p-2}\int_Mf^2dv_g
\ee
Applying Lemma 4.1(vi) we have,
\Bea(|R|^p)'_g(fg)&=&-2pc|R|^{p-2}(trD\dg fg-\la trfg+(n-1)ctrfg)\\
&=&-2pc|R|^{p-2}(\la f-n\la f+n(n-1)cf)\\
&=&-2p|R|^{p-2}(n-1)c(ncf-\la f)
\Eea
Consequently,
\Bea tr\big((|R|^{p-2})'(fg)g\big)&=& -2cn(n-1)(p-2)|R|^{p-4}(ncf-\la f)\\
&=&\frac{(p-2)}{c}|R|^{p-2}(\la f-ncf)
\Eea
Hence,
\bea-\frac{p}{n} \|R\|^2\langle(|R|^{p-2})'(fg)g,fg\rangle
=-2pc(p-2)(n-1)\|R\|^{p-2}[\| df\|^2-nc\int_Mf^2dv_g]
\eea
and
\bea\frac{1}{2}\langle (|R|^p)'g,fg\rangle
&=&-pnc(n-1)\|R\|^{p-2}\int_M(-f\la f+ncf^2)dv_g\\
\nonumber&=&npc(n-1)\|R\|^{p-2}[\| df\|^2
-nc\int_Mf^2dv_g]
\eea
From Lemma 4.1(i),
\Bea tr(\Rt)'(fg)=-2c^2n(n-1)f+4c(n-1)\la f
\Eea
Therefore,
\be-p\|R\|^{p-2}\langle(\Rt)'(fg),fg\rangle
=-2cp(n-1)\|R\|^{p-2}[2\|df\|^2-cn\int_M f^2dv_g]
\ee
Next, we compute the 4th term in expression of $H$ in (4.1). By a straightforward computation we have the following identity,
\Bea Dd^Dh(x,y,z,w)=D^2_{x,z}h(y,w)-D^2_{x,w}h(y,z)
\Eea
This yields,
\Bea(R,Dd^Dfg)&=&2\sum R_{ijkl}Dd^Dfg_{ijkl}\\
&=&2\sum R_{ijij}((D^2_{ii}fg)_{jj}-(D^2_{ij}fg)_{ij})\\
&=&2c\sum(trDdtrfg+trD\dg fg)\\
&=&-2c(n-1)\la f
\Eea
Therefore,
\bea -p\langle(|R|^{p-2})'R,Dd^Dfg\rangle
&=&-p\int_M(|R|^{p-2})'_{fg}(fg)(R,Dd^Dfg)dv_g\\
\nonumber&=&4p(n-1)^2(p-2)c^2\|R\|^{p-4}[\|\la f\|^2-nc\|df\|^2]
\eea
Next using Lemma 4.1 (v) we have,
\Bea tr\dd d^Dfg &=&2trD^*D(fg)-2trD\dg (fg)\\
&=&2\big(\la (tr(fg))+trDdf\big)\\
&=&2(n-1)\la f
\Eea
This identity combining with Lemma 4.1 (ii) implies that
\Bea\langle \dd W_{(fg)},fg\rangle
&=&c(n-2)\int_M(tr\dd d^D fg)fdv_g\\
&&+2nc\int_M(trDdf)fdv_g+2n^2c\int_M f\la f dv_g\\
&=& 4c(n-1)^2\|df\|^2
\Eea
Therefore,
 \be-p\|R\|^{p-2}\langle \dd W_{(fg)},fg\rangle
= -4pc(n-1)^2\|R\|^{p-2}\| df\|^2
\ee
Next, we compute the remaining term appearing in the expression of the Hessian.
From Lemma 4.1(iv) we obtain,
\Bea  \bar{r} &=& \frac{1}{2}\{2(n-1)cfg -2\dg^* \dg fg
-Ddtrfg+D^*Dfg \}\\
&=&\frac{1}{2}\{2c(n-1)fg+2Ddf-nDdf+\la fg\}\\
&=&\frac{1}{2}\{2c(n-1)fg -(n-2)Ddf +\la fg\}\Eea
By a simple calculation using lemma 4.1 (v) we have,
\Bea \dd d^Dfg=2(\la fg +Ddf) \Eea
Therefore,
\Bea \langle \bar{r} , \dd d^Dfg\rangle
&=& (2n-3)\langle \la f,\la f\rangle-(n-2)\langle Ddf,Ddf\rangle
+2c(n-1)^2\langle df,df\rangle\\
&=&(n-1)\langle Ddf,Ddf\rangle+(n-1)(4n-5)c\langle df,df \rangle
\Eea
Using Bochner-Weitzenb\"{o}k formula on the space of one forms we have,
$$\la df=D^*Ddf+(n-1)cdf$$
This implies,
\Bea
\|\la f\|^2=\langle \de df,\de df\rangle=\langle \la df,df\rangle
=\|Ddf\|^2+(n-1)c\|df\|^2
\Eea
Therefore,
 \be \langle \bar{r} , \dd d^Dfg\rangle=(n-1)\|\la f\|^2+c(n-1)(3n-4)\|df\|^2
\ee

Hence combining all the equations from (4.3) to (4.9) we have,
\Bea  H(fg, fg)=p\|R\|^{p-2}\big( a\|\la f\|^2-bc\langle \la f,f\rangle+dc^2\|f\|^2\big)
\Eea
where
\Bea && a=(n-1)+2(p-2)(1-\frac{1}{n})\\
     && b=4(n-1)(p-1)\\
     && d=n(n-1)(2p-n)
\Eea
Consider the polynomial, $q(x)=ax^2-bx+d$. Suppose $f$ be an eigenfunction of the Laplacian corresponding to the eigenvalue $\lambda c$. Then
$$ H(fg, fg)= q(\lambda)c^2\|f\|^2$$
To prove our claim it is sufficient to prove that $q(\lambda)>0$.
Notice that
$$q(x)=(x-n)(ax-\frac{d}{n})$$
Let $c>0$. Since $\frac{d}{an}<n$ and the first eigenvalue $c\lambda_1$ of $\la$ satisfies $\lambda_1\geq n$ we have that $q(\lambda)\geq 0$. $q(\lambda)=0$ if and only if $\lambda=\lambda_1=n$. This implies that $(M,g)$ is a sphere with the standard metric. In this case, the eigenfunctions are the first order spherical harmonics. These functions satisfy, $ \dg ^* df=Ddf=-fg$. Hence the proof follows.

If $c<0$ then the proof immediately follows from the expression of $H(fg,fg)$.
\hfill
$\square$
\\

Next to obtain the stability of $\mathcal{R}_p$ for space forms it is sufficient to prove that
$ H(h,fg)=0$ for any $h$ be a TT-tensor and $f\in C^{\infty}M$. From \cite{BE} the decomposition (1.1) is preserved by the rough Laplacian. Hence, it is easy to see from the Lemma 4.1 that
$$tr((\Rt)'(h))=tr(\dd d^Dh)=tr(\dd W_h)=tr(\bar{r}_h)=0$$
and  $$\dg (\bar{r}_h)=0$$
This implies that $ tr(\dd d^D \bar{r}_h)=0$. Lemma 4.1 (vi) implies that $(|R|^p)'(h)$ is also zero.\\
Hence, $$ H(h,fg)=0.$$
\hfill
$\square$\\
\section{Second Variation at product of space forms} In this section we prove the stability of $\mathcal{R}_p$ for product of space forms of same type for certain values of $p.$ Let $(M_1^m,g_1)$ and $(M_2^m,g_2)$ be two closed Riemannian manifolds with dimension $m\geq 3$ and constant sectional curvature $c\neq 0$. Let $(M,g)=(M_1\times M_2, g_1+g_2)$.

From [BA] Lemma 4.57 (ii), we have the following orthogonal decomposition of $T_g\mathcal{M}_1$.
\bea T_g\mathcal{M}_1= {\rm Im}\dg^*\oplus C^{\infty}(M)\oplus(\dg^{-1}(0)\cap tr_g^{-1}(0))
\eea
Let $E_1=\{e_1$ ,$e_2$, ...,$e_m\}$ and $E_2=\{e_{m+1}$,....,$e_{2m}\}$ denote normal basis at some points $p_1$ and $p_2$ corresponding to $(M_1^m,g_1)$ and $(M_2^m,g_2)$ respectively. The curvature $R$ satisfies the following properties,\\
\\
(R1) $R(e_i,e_j,e_i,e_j)=-R(e_i,e_j,e_j,e_i)=c$, when $\{e_i,e_j\}\subset E_k$, $k=1,2$.\\
\\
(R2) $R(e_m,e_n,e_i,e_j)=0$,  otherwise.\\

A traceless symmetric tensor splits as
\be h=h_1+fg_1+\th+h_2-fg_2
\ee
where, $h_1$ is tangent to the first factor, $h_2$ is tangent to the second factor and $\th$ is non-zero only for the mixed set of vectors and $f\in C^{\infty}(M_1\times M_1)$. This decomposition is preserved by the rough Laplacian and
$$tr (h_1)=tr (h_2)=tr(\th)=0$$
Let $h\in C^{\infty}(M).g\oplus(\dg^{-1}(0)\cap tr^{-1}(0))$. Then we have that $\dg h=-\frac{1}{n}dtrh$. Moreover, if $h$ is a TT-tensor then $$\dg^*\dg h_1=\dg^*\dg h_2=\dg^*\dg \th=0$$
To prove the theorem, we need the following Lemma.
\begin{lem}
\Bea &&\Rt'(\th)=4\dg^*\dg\th+D^*D\th\\
&&\Rt'(h_1)=2(m+1)c^2h_1+2cD^*Dh_1-4c\dg^*\dg h_1\\
&&\Rt'(fg_1)=-2(m-1)c^2fg_1+2c[\la_1 fg_1-(m-2)\dg^*df_1]
\Eea
where $df_1$ is the component of $df$ along the first factor.
\end{lem}
\begin{proof} From the proof of the Lemma 4.1 (i),
 \Bea \Rt' (h)_{pq}&=&-\sum_{m,n,i,j}h_{mn}\left(R_{pmij}R_{qnij}+R_{pimj}R_{qinj}+R_{pijm}R_{qijn}\right)\\
&&+\sum_{i,j,k} R'(h)_{pijk}R_{qijk}+\sum_{i,j,k}R_{pijk}R'(h)_{qijk}
\Eea
Using (R1) and (R2) we have that $$\sum_{m,n,i,j} \th_{mn}\big(R_{pmij}R_{qnij}+R_{pimj}R_{qinj}+R_{pijm}R_{qijn}\big)=0$$
and $\sum_{i,j,k} R'(\th)_{pijk}R_{qijk}$ is non-zero only if $\{e_p,e_i,e_j,e_k\}\subset E_k$, $k=1,2$.
Now,
  \Bea 2\sum_i R'(\th)_{piqi}&=&[(D^2_{iq}\th)_{pi}+(D^2_{pi}\th)_{qi}-(D^2_{pq}\th)_{ii}-(D^2_{ii}\th)_{pq}+\th_{ij}R_{piqj} -\th_{qj}R_{piij}]
  \Eea
It is clear from the above expression that $\Rt'(\th)_1=\Rt'(\th)_2=0$. Hence,
$$\sum R'(\th)_{pijk}R_{qijk}=c^2\dg^*\dg \th+\frac{1}{2}D^*D\th$$
Therefore,
$$\Rt'(\th)=4\dg^*\dg \th+D^*D\th$$

Next using (R1) and (R2) again we have,
\Bea \sum_{m,n,i,j} h_{1mn}\big(R_{pmij}R_{qnij}+R_{pimj}R_{qinj}+R_{pijm}R_{qijn}\big)=2(m-3)c^2h_{1pq}
\Eea
If $e_p,e_q\in E_2$, a simple computation shows that
 $\sum_{i\in E_2}R'(h_1)_{piqi}=0$\\
If $e_p,e_q\in E_1$, then
\Bea
\sum R'(h_1)_{pijk}R_{qijk}&=&2c\sum_{i\in E_1}R'(h_1)_{piqi}=c[D^*Dh_1+2(m-1)ch_1-2\dg^*\dg h_1]
\Eea
Hence,
$$\Rt'(h_1)_{pq}=2(m+1)c^2h_{1pq}+2cD^*Dh_{1pq}-4\dg^*\dg h_1$$
\\
Similarly,
$$\Rt'(fg_1)=2(m+1)c^2fg_1-4mc^2fg_1+2c[-mDdf_1+2\dg^*df_1+\la_1 fg_1] $$
\end{proof}
Next two lemma follow from the proof of Lemma 5.1 and Lemma 4.1.
\begin{lem}\Bea && \bar{r}_{\th}=\frac{1}{2}[D^*D\th-2\dg^*\dg\th]\\
&&\bar{r}_{h_1}=\frac{1}{2}[2c(m-1)h_1+D^*Dh_1-2\dg^*\dg h_1]\\
&&\bar{r}_{fg_1}=\frac{1}{2}[2c(m-1)fg_1+2\dg^*df_1-mDdf+\la fg_1]
\Eea
\end{lem}
\begin{lem}\Bea &&(|R|^p)'\th=0\\
&&(|R|^p)'h_1=-4pc|R|^{p-2}tr(\dg^*\dg h_1)\\
&&(|R|^p)'(fg_1)=2cp(m-1)|R|^{p-2}\big(\la_1 f-mcf\big)
\Eea
\end{lem}
\begin{lem}\Bea&&\dd d^D\th=2D^*D\th+2c(m-1)\th-2\dg^*\dg \th\\
&&\dd d^Dh_1=2D^*Dh_1+2mch_1-2\dg^*\dg h_1\\
&&\dd d^Dfg_1=2\la fg_1+2\dg^*df_1
\Eea
\end{lem}
The proof easily follows from the proof of Lemma 4.1 (v).
\begin{lem}\Bea&&(\dd W_{\th})_k=0, for k=1,2\\
&&\langle W_{\th}, d^D \th\rangle = (m-1)c\|d^D \th\|^2+\frac{c}{2}K, \ \ where \ \  0\leq K\leq \|d^D\th\|^2\\
&&\dd W_{h_1}= c(m-2)\dd d^Dh_1\\
&&\dd W_{fg_1}=(m-1)c\dd d^D(fg_1)+2cm\la_1 fg_1+2cm\dg^*df_1
\Eea
 \end{lem}
\begin{proof} From the proof of Lemma 4.1 (ii) we have that for any $h$, $\a\in S^2(T^*M)$,
\Bea \sum W_{hjkl}d^D\a_{jkl}
 =2\sum \big(R_{ij\Pi_{ik}l}-R_{li\Pi_{ik}j}
-R_{lii\Pi_{kj}}\big)(d^D\a)_{jkl}
\Eea
Now consider $\th$.
\bea \sum R_{ij\Pi_{ki}l}d^D\a_{jkl}&=& \sum C_{kim}R_{ijml}d^D\a_{jkl}\\
\nonumber&=& c\sum_{i,j\in E_1} C_{kii}d^D\a_{jkj}-c\sum_{j,l\in E_1} C_{klj}d^D\a_{jkl}\\
\nonumber&&+c\sum_{i,j\in E_2} C_{kii}d^D\a_{jkj}-c\sum_{j,l\in E_2} C_{klj}d^D\a_{jkl}
\eea
$\sum_{i\in E_1} C_{kii}=dtr_{g_1}(\th)_k=0$,\\
\\
As we have seen in Lemma 4.1(ii), $\sum_{j,l\in E_1} C_{klj}d^D\a_{jkl}=0$.\\
\\
Similarly, the last two terms of (5.3) are also zero.\\
Next,
\Bea \sum R_{lii\Pi_{kj}}(d^D\a)_{jkl}&=& C_{\th kjl}R_{liil}d^D\a_{jkl}\\
&=&-(m-1)c\sum C_{\th kjl}d^D\a_{jkl}\\
&=&-\frac{c(m-1)}{2}d^D\th_{jkl}d^D\a_{jkl}
\Eea

\Bea \sum R_{li\Pi_{ik}j}d^D\a_{jkl}
&=&\sum C_{\th ikm}R_{limj}d^D\a_{jkl}\\
&=&\sum C_{\th ikl}R_{lili}d^D\a_{ikl}+\sum C_{\th iki}R_{liil}d^D\a_{lkl}\\
&=&c\sum_{l,i\in E_1}C_{\th ikl}d^D\a_{ikl}+c\sum_{l,i\in E_2}C_{\th ikl}d^D\a_{ikl}
\Eea
Clearly for $\a=h_1$ or $\a=h_2$, the above expression is zero.
Let $\a=\th$.\\
Then by a simple calculation we have,
\Bea \sum_{l,i\in E_1}C_{\th ikl}d^D\th_{ikl}
=-\frac{1}{4}\sum_{i,l\in E_1}|d^D\th_{ikl}|^2
\Eea
and
\Bea \sum_{l,i\in E_1}C_{\th ikl}d^D\th_{ikl}
=-\frac{1}{4}\sum_{i,l\in E_1}|d^D\th_{ikl}|^2
\Eea
Suppose, $$ K=\frac{1}{4}\int_M\big( \sum_{i,l\in E_1}|d^D\th_{ikl}|^2
+\sum_{i,l\in E_1}|d^D\th_{ikl}|^2 \big)dv_g$$
 Then, $0\leq K\leq \frac{1}{4} \|d^D\th\|^2$.\\
Hence the result follows.

Next, consider $h_1$. It is easy to see using the formula for $C_{h_1}$ that $C_{h_1ijk}$ is zero if $\{e_i$, $e_j$, $e_k\}$ intersects $E_2$. Using this and following the similar computation as in Lemma 4.1 (ii) we get the result.

Now, consider $h=fg_1$. In this case, a straightforward calculation gives,
\Bea  \sum \big(R_{ij\Pi_{ik}l}-R_{li\Pi_{ik}j}\big)d^D\a_{jkl}= 2\sum C_{kii}R_{ijij}d^D\a_{jkj}+\sum C_{kij}R_{ijij}d^D\a_{ikj}
\Eea
Since $C_{kii}=0$, when $e_i\in E_2$,
\Bea2\sum C_{kii}R_{ijij}d^D\a_{jkj}&=&2c\sum_{i,j\in E_1}C_{kii}d^D\a_{jkj}\\
&=& c(m-1)\sum df_k(dtr\a_{1k}+\dg\a_{1k})
\Eea
Since $C_{kij}=\frac{1}{2}(df_kg_{ij}+df_ig_{kj}-df_jg_{ik})$,
\Bea\sum C_{kij}R_{ijij}d^D\a_{ikj}=c\sum_{i,j\in E_1}df_j(dtr\a_{1j}+\dg\a_{1j})
\Eea
Therefore,
$$\sum \big(R_{ij\Pi_{ik}l}-R_{li\Pi_{ik}j}\big)d^D\a_{jkl}=cm\sum df_k(dtr\a_{1k}-\dg\a_{1k})$$
\Bea \sum R_{lii\Pi_{kj}}(d^D\a)_{jkl}= -\frac{c}{2}(m-1)d^D(fg_1)_{jkl}d^D\a_{jkl}
\Eea
Hence,
$$\dd W_{fg_1}=(m-1)c\dd d^D(fg_1)+2cm\la_1 fg_1+2cm\dg^*df_1$$
\end{proof}
Next we study the stability of $\mathcal{R}_p$ for product of space forms. First we study the action of $H$ on TT-tensors.
\subsection{Transverse-traceless Variations:}  Consider $h\in \dg^{-1}(0)\cap tr^{-1}(0)$. Suppose $h=h_1+\th+h_2+fg_1-fg_2$. It is easy to see using the above lemma that
$$ H(h_1,h_2)= H(h_1,\th)=H(h_2,\th)=0$$
and
\Bea && H(h_1,h_1)=p|R|^{p-2}[\|D^*Dh_1\|^2+mc\|Dh_1\|^2+2(m-2)c^2\|h_1\|^2]\\
&& H(h_2,h_2) =p|R|^{p-2}[\|D^*Dh_2\|^2+mc\|Dh_2\|^2+2(m-2)c^2\|h_2\|^2]\\
&& H(\th,\th) =p|R|^{p-2}[\|D^*D\th\|^2+c(m-1)\|D\th\|^2+2c^2(m-1)\|\th\|^2-\frac{c}{2}K]
\Eea
Using similar arguments as in section 4.1, we have, $\epsilon_1$ and $\epsilon_2$ such that $ H(h_1,h_1)\geq \epsilon_1\|h_1\|^2$ and $H(h_2,h_2)\geq \epsilon_2\|h_2\|^2$.
Now, using the estimate for $K$ given in Lemma 5.1(v), we have,
\Bea H(\th,\th) \geq p|R|^{p-2}[\|D^*D\th\|^2+c(m-\frac{5}{4})\|D\th\|^2+\frac{7}{4}c^2(m-1)\|\th\|^2]
\Eea
If $c>0$, then it is clear from the above expression that
$$ H(\th,\th) \geq \epsilon_3\|\th\|^2$$
Suppose $c<0$, then $c(m-\frac{5}{4})\geq c(m-1)$. Now, $\|d^D\th\|^2\geq 0$ implies that
  $$\|D^*D\th\|^2+c(m-1)\|D\th\|^2\geq 0$$
Hence,
$$ H(\th,\th) \geq \epsilon_3\|\th\|^2$$
Using bi-linearity of $H$ we have
\be H(h,h)= H(h_1, h_1)+ H(h_2,h_2)+ H(\th,\th)+ H(fg_1,fg_1)+ H(fg_2,fg_2)+H(fg_1,fg_2)
\ee
Next we shall compute the remaining terms of (5.4).
From Lemma 5.1 we have,
\Bea \langle (\Rt)'(fg_1), fg_1\rangle &=&-2(m-1)c^2\|fg_1\|^2+2c[\langle \la_1fg_1,fg_1\rangle -(m-2)\langle \dg^*df_1,fg_1\rangle]\\
&=& -2c^2m(m-1)\|f\|^2+4c(m-1)\|df_1\|^2
\Eea
where $df_1$ is the component of $df$ along the tangent space of $M_1$.
\Bea \langle \bar{r}_{fg_1}, \dd d^Dfg_1\rangle
&=& \langle 2c(m-1)fg_1+2\dg^*df_1-mDdf+\la f g_1, \la fg_1+\dg^*df_1\rangle\\
&=& 2cm(m-1)\|df\|^2+(m-3)\langle \la_1f,\la f\rangle+m\|\la f\|^2\\
&&-(m-2)\|\dg^*df_1\|^2-2c(m-1)\|df_1\|^2\\
&=& 2cm(m-1)\|df\|^2+(2m-3)\|\la_1f\|^2+3(m-1)\langle\la_1 f,\la_2 f\rangle+m\|\la_2 f\|^2\\
&&-(m-2)\|\dg^*df_1\|^2-2c(m-1)\|df_1\|^2
\Eea
Using Bochner-Weitzenb\"{o}k formula on the space of one forms we have,
$$\la df_1=D^*Ddf_1+(m-1)cdf_1$$
Next, a simple calculation yields the following identity for a one-form $\omega$,
\be 2\dg\dg^*\omega+\delta d\omega=2D^*D\omega\ee
Using this identity we have,
$$\|\dg^*df_1\|^2=\langle \dg\dg^*(df_1), df_1\rangle=\|\la_1f\|^2-c(m-1)\|df_1\|^2$$
Therefore,
\Bea\langle \bar{r}_{fg_1}, \dd d^Dfg_1\rangle&=&2cm(m-1)\|df\|^2+(m-1)\|\la_1f\|^2+c(m-1)(m-4)\|df_1\|^2\\
&&+3(m-1)\langle\la_1 f,\la_2 f\rangle+m\|\la_2 f\|^2
\Eea
Next,
\Bea\langle \dd W_{fg_1}, fg_1\rangle
&=& 2c(m-1)[\langle \la fg_1+\dg^*df_1, fg_1\rangle+2cm\langle \la_1 fg_1,fg_1\rangle+2cm\langle \dg^*df_1,fg_1\rangle\\
&=&2cm(m-1)\|df\|^2+2c(m-1)^2\|df_1\|^2
\Eea
\Bea (R, Dd^Dfg_1)&=&
2c\sum_{i,j\in E_1}(Dd^Dfg_1)_{ijij}+2c\sum_{i,j\in E_2}(Dd^Dfg_1)_{ijij}\\
&=&2c\sum_{i,j\in E_1}\big((D^2_{ii}fg_1)_{jj}-(D^2_{ij}fg_1)_{ij}\big)\\
&=&-2c(m-1)\la_1f
\Eea
Therefore,
\Bea \langle (|R|^{p-2})'(fg_1)R, Dd^Dfg_1\rangle
&=& -4c^2(p-2)(m-1)^2|R|^{p-4}\langle \la_1f-mcf, \la_1f\rangle\\
&=& -(p-2)(1-\frac{1}{m})|R|^{p-2}[\|\la_1 f\|^2-mc\|df_1\|^2]
\Eea
\Bea \frac{1}{n}|R|^2\langle (|R|^{p-2})'(fg_1).(g_1+g_2),fg_1\rangle
&=&c(p-2)(1-\frac{1}{m})|R|^{p-2}\langle (\la_1f-mcf)g_1, fg_1\rangle\\
&=& c(p-2)(m-1)|R|^{p-2}[\|df_1\|^2-mc\|f\|^2]
\Eea
\Bea \frac{1}{2}\langle (|R|^p)'(fg_1)g_1,fg_1\rangle
&=& mpc(m-1)|R|^{p-2}[\|df_1\|^2-mc\|f\|^2]
\Eea
Combining all these results, we have,
\Bea
H(fg_1,fg_1)&=& p(m-1)|R|^{p-2}[a\|\la_1 f\|^2-bc\|df_1\|^2+dc^2\|f\|^2]\\
&&+p|R|^{p-2}[3(m-1)\langle\la_1f,\la_2f\rangle+m\|\la_2f\|^2]
\Eea
where, $a=\frac{1}{m}(m+p-2)$, $b=2(p+1)$, $d=m(p-m+2)$.\\
\\
Performing similar computation we have,
\Bea H(fg_1, fg_2)&=&p|R|^{p-2}[2 \langle\la_1f,\la_2f\rangle +m(m-1)c\|df\|^2-m^2(m-1)c^2\|f\|^2]\\
&&+p(p-2)(m-1)|R|^{p-2}[\frac{1}{m}\langle \la_1f,\la_2f\rangle-c\|df\|^2 +mc^2\|f\|^2]
\Eea
and
\Bea H(fg_2,fg_2)&=& p(m-1)|R|^{p-2}[a\|\la_2 f\|^2-bc\|df_2\|^2+dc^2\|f\|^2]\\
&&+p|R|^{p-2}[3(m-1)\langle\la_1f,\la_2f\rangle+m\|\la_1f\|^2]
\Eea
Therefore,
\Bea H(fg_1-fg_2,fg_1-fg_2)&=& H(fg_1,fg_1)-2H(fg_1,fg_2)+H(fg_2,fg_2)\\
&=&p|R|^{p-2}[a_1\|\la_1 f\|^2+a_1\|\la_2f\|^2+b_1c\|df\|^2+2d_1c^2\|f\|^2]\\
&&+p|R|^{p-2}u_1\langle \la_1 f,\la_2 f\rangle
\Eea
where
\Bea && a_1=(m-1)a+m\\
     && u_1=\frac{2}{m}\{3m^2-3m-2-p(m-1)\}\\
     && b_1=-2(m-1)(m+3)\\
     &&d_1=4m(m-1)
\Eea
{\bf Case 1: } $c>0$. We know that the first eigenvalue of the Laplacian is greater than $mc$. Suppose, $f$ be an eigenfunction corresponding to the eigenvalue $c\lambda$ of the Laplacian of $(M_1\times M_2,g_1+g_2)$. Then $f=f_1f_2$ and $\lambda=\mu_1+\mu_2$ where $f_1$ and $f_2$ are eigenfunctions of the Laplacian for $(M_1,g_1)$ and $(M_2,g_2)$ corresponding to the eigenvalues $c\mu_1$ and $c\mu_2$ respectively. Therefore,
 $$ \langle\la_1 f,\la_2 f\rangle=c^2\mu_1\mu_2 |f|^2$$
Since $u_1\geq 0$ for $p\leq 2m$, we have,
\Bea H(fg_1-fg_2,fg_1-fg)_2&&\geq p|R|^{p-2}[a_1\|\la_1 f\|^2+a_1\|\la_2f\|^2+b_1c\|df\|^2+d_1c^2\|f\|^2]\\
&&\geq p|R|^{p-2}[a_1\|\la_1 f\|^2+b_1c\|df_1\|^2+d_1c^2\|f\|^2]\\
&&+p|R|^{p-2}[a_1\|\la_2f\|^2+b_1c\|df_1\|^2+d_1c^2\|f\|^2]
\Eea
Now consider the polynomial
$$q_1(x)=a_1x^2+b_1x+ d_1$$
Note that,
$$ H(fg_1-fg_2,fg_1-fg_2)\geq pc^2|R|^{p-2}(q_1(\mu_1)+q_1(\mu_2))\|f\|^2$$
So, it is sufficient to prove that $q_1(x)>0$ for $x\geq m$.
$$q_1'(x)=2a_1x+b_1$$
By a simple computation we have that, $q_1'(x)>0$ for $x\geq m$ and $q_1(m)>0$.\\
 This completes the proof.\\
\\
{\bf Case 2:}  $c<0$. Since $b_1<0$ and $u_1\langle \la_1f,\la_2f\rangle>0$ we have that
\Bea H(fg_1-fg_2,fg_1-fg_2)\geq 2p|R|^{p-2}d_1c^2\|f\|^2\\
\Eea
It is easy to see from the Lemma 5.1 that $H$ is diagonalizable by the decomposition (5.1). Therefore to complete the proof it is sufficient to show that there exists an $\epsilon_3>0$ such that $H(fg,fg)\geq \epsilon_3\|fg\|^2$.
\subsection{Conformal Variations:} Consider $f$ in $C^{\infty}(M_1\times M_2)$.
Using the computations in 5.1 we have,
\Bea  H(fg_1+fg_2, fg_1+fg_2)&=& H(fg_1,fg_1)+2 H(fg_1,fg_2)+ H(fg_2,fg_2)\\
&=&p|R|^{p-2}[a_2\|\la f\|^2+u_2\langle \la_1 f,\la_2 f\rangle+b_2c\|df\|^2+d_2c^2\|f\|^2]
\Eea
where
\Bea &&a_2=a_1,u_2=2m\\
     && b_2=-2(m-1)(2p-m-1)\\
      &&d_2=4m(m-1)(p-m)
\Eea
Since $u_2>0$,
$$ H(fg_1+fg_2, fg_1+fg_2)\geq p|R|^{p-2}[a_2\|\la f\|^2+b_2c\|df\|^2+d_2c^2\|f\|^2]  $$
{\bf Case1:} $c>0$.
Consider the polynomial
$$q_2(\lambda)=a_2\lambda^2+b_2\lambda+d_2$$
A simple computation gives if $p\leq 2m$, then $2a_2m+b_2>0$ and $q_2(m)>0$. Using the argument as in 5.1 the proof follows.\\
{\bf Case2:} $c<0$. When $p\geq m$ it is easy to see that $b_2<0$ and $d_2>0$. Therefore, $q_2(\lambda)>0$.
This completes the proof.
\hfill
$\square$
\section{Local minimization} To obtain local minimization property for $\mathcal{R}_p$, we follow the techniques used in [GV]. First we consider the scale-invariant functional defined by,
$$\tR(g)=(V(g))^{\frac {2p}{n}-1}.\mathcal{R}_p(g) $$
A simple calculation shows that,
\Bea \nabla \tR (g)=V^{\frac{2p}{n}-1}\nabla\mathcal{R}_p(g)+(\frac{p}{n}-\frac{1}{2})V^{\frac{2p}{n}-2}\mathcal{R}_p(g)g
\Eea
It is easy to see that $g$ is a critical metric for $\mathcal{R}_{p|\mathcal{M}_1}$ if and only if it is critical for $\tR$. Let $\tilde{H}_{\gt}$ denote the second derivative of $\tR$ at $\gt$. Recall that
$$\mathcal{W}=({\rm Im}\dg^*)^{\perp}\cap T_g\mathcal{M}_1$$
Let $(M,g)$ be a critical point for $\tR$. $(M,g)$ is {\it$L^{2,2}$- stable} for $\tR$, if there exists $\epsilon>0$ such that for any $h \in \mathcal{W}$,
$$ \tilde{H}_{g}(h,h)\geq \epsilon\|h\|^2_{L^{2,2}}$$
where
$$\|h\|^2_{L^{2,2}}=\|D^2h\|^2+\|Dh\|^2+\|h\|^2$$
\begin{pro}
Let $(M,g)$ be a closed Riemannian manifold. If $(M,g)$ is $L^{2,2}$-stable for $\tR$ then it is a strict local minimizer for $\tR$.
\end{pro}
We need the following lemma to prove the proposition.
\begin{lem} For each metric $\tilde{g}=g+\theta_1$ in a sufficiently small $C^{l+1,\a}$-neighborhood of $g$ $(l\geq 1)$, there is a $C^{l+2,\a}$-diffeomorphism $\phi:M\rightarrow M$ and a constant c such that
\Bea \tilde{\theta}=e^c\phi^*\tilde{g}-g
\Eea
satisfies
\Bea \dg\tilde{\theta}=0 \ \ and \ \
 \int tr(\tilde{\theta})dv_g=0
\Eea
Moreover, we have the estimate
\Bea \|\tilde{\theta}\|_{C^{l+1,\a}}\leq C\|\theta_1\|_{C^{l+1,\a}}
\Eea
\end{lem}
\begin{proof}: Consider the operator $$\dg\dg^*:T^*M\to T^*M$$
Since this is an elliptic operator, the lemma follows from the proof of Lemma 2.10 in [GV].
\end{proof}
We denote by $A*B$ any tensor field which is a real linear combination of tensor fields, each formed by starting with the tensor field $A\otimes B$, using the metric to switch the type of any number of $T^*M$ components to $TM$ components, or vice versa taking any number of contractions, and switching any number of components in the product. For any two tensor $A$ and $B$ we have, $|A*B|\leq C|A||B|$ for some constant $C$ which will depend neither on $A$ nor $B$.
\begin{lem}
There exists a neighborhood $V$ of $g$ and a positive constant $C_1$ such that for any $\gt\in V$,
\be |\tR(\gt)-\tR(g)|\leq C_1\|\gt-g\|^2_{C^{2,\a}}\ee
\end{lem}
\begin{proof} Let $\gt=g+\theta$ and $T$ be a tensor. We have the following relation between the connection of $g$ and $\gt$,
\be D_{g+\theta}T=D_gT+(g+\theta)^{-1}*D_g\theta*T
\ee
The curvature of $g$ and $\gt$ related by,
\be R(g+\theta)=R(g)+(g+\theta)^{-1}*D^2\theta+(g+\theta)^{-2}*(D \theta*D\theta)\ee
We also have the following formula.
\be(g+\theta)^{-1}-g^{-1}=-g^{-1}(g+\theta)^{-1}\theta\ee
The lemma follows by using some standard techniques and the above equations.
\end{proof}
\begin{lem} Let $g$ be a Riemannian metric on $M$ with unit volume. There exists a neighborhood $U$ of $g$ in $\mathcal{M}_1$ such that for any $\gt\in U$ and $h\in \mathcal{W}$,
  $$|\tilde{H}_{\gt}(h,h)-\tilde{H}_g(h,h)|\leq C\|\gt-g\|^4_{C^{2,\a}}\|h\|_{L^{2,2}}$$
\end{lem}
\begin{proof} By a straight forward computation we have,
\Bea\tilde{H}_g&=&-2\langle \nabla \tR,h\circ h\rangle_g+\langle (\nabla \tR)'(h),h\rangle_g \\
&=& 2[p\langle |R|^{p-2}R, Dd^D(h\circ h)\rangle +p\langle |R|^{p-2}\tR, h\circ h\rangle-\frac{1}{2}\langle |R|^p,|h|^2\rangle]\\
&&+\langle (\nabla \mathcal{R}_p)'(h),h\rangle-(\frac{p}{n}-\frac{1}{2})\mathcal{R}_p(g)\|h\|^2
\Eea
We observe from the expression of $\tilde{H}$ that $\tilde{H}(g)=\int_M f|R|^{p-2}dv_g$, where $f\in C^\infty(M)$ and $\int_M fdv_g$ is the second derivative of $\tilde{\mathcal{R}}_2$. Using the previous lemma it is sufficient to prove the lemma for the second derivative for $\tilde{\mathcal{R}}_2$.

Suppose $\tilde{H}$ denote the second derivative of $\tilde{\mathcal{R}}_2$. We have,
\Bea &&(R, Dd^D(h\circ h))=g^{-1}*g^{-1}*g^{-1}*g^{-1}*R*(D^2h+ Dh*Dh)\\
&& (\Rt , h\circ h)=g^{-1}*g^{-1}*g^{-1}*g^{-1}*R*R\\
&&(\bar{R}_h, Dd^Dh)= g^{-1}*g^{-1}*g^{-1}*g^{-1}(D^2h*D^2h+h*R)\\
&& \langle W_h, d^Dh\rangle= \int_M (g^{-1}*g^{-1}*g^{-1}*g^{-1}*R*Dh*Dh) dv_g\\
&& ((\Rt)'(h),h)=g^{-1}*g^{-1}*g^{-1}*g^{-1}*R*h*(R*h + D^2h)\\
&& (|R|^p)'(h)= |R|^{p-2}*g^{-1}*g^{-1}*g^{-1}*g^{-1}*(R*D^2h+ R*R*h)\\
&& \langle (\dd)'(h) D^*(R)= g^{-1}*g^{-1}*g^{-1}*g^{-1}*d^2h*h*R
\Eea
Combining above equations we obtain the required result.
\end{proof}
Proof of Proposition 3:  Choose a neighborhood $U$ of $g$ in $C^{2,\a}$-topology such that the following conditions hold.

(i) Lemma 6.1 and 6.3 hold on U.

(ii) Let $\tilde{g}=g+\theta_1\in U$. Then using Lemma 6.1 we have, $\tilde{\theta}$ satisfying the conditions given in Lemma 6.1. We can assume $g+t\tilde{\theta}\in U$ for all $t\in [0,1]$.

(iii) Since $g$ is $L^{2,2}$-stable, we can assume that for any $\gt\in U$ with $V(\gt)=V(g)$, $\tilde{H}_g(h,h)>0$ for all $h\in \mathcal{W}$.

We have,
 $$\tR(g+\tilde{\theta})=\tR(e^c\phi^*\gt)=\tR(\phi^*\gt)=\tR(\gt)=\tR(g+\theta_1)$$
Define
$$\gamma(t)=g+t\tilde{\theta}$$
$\gamma(t)\in U$ for $t\in [0,1]$.
Let
\Bea a(t)=\tR(\gamma(t))
\Eea
Then $a(0)=\tR(g)$, $a(1)=\tR(g+\tilde{\theta})$ and $a'(0)=0$.
Since $\tilde{\theta}\in \mathcal{W}$
\Bea a^{\prime\prime}(t)= \tilde{H}_{\gamma(t)}(\tilde{\theta},\tilde{\theta})
>0
\Eea
Therefore,
\Bea a(1)-a(0)=\int^1_0 \int^1_0 a^{\prime\prime}(st)dsdt>0
\Eea
If $\tR(\gt)=\tR(g)$, then $\tilde{\theta}=0$. Hence $\gt$ is isometric to $g$. This completes the proof.
\hfill
$\square$
\\
\\
The following corollary is an immediate consequence of this proposition.
\begin{cor} Let $(M,g)$ be a closed Riemannian manifold with dimension $n\geq 3$.
If $(M,g)$ is one of the following then $g$ is strict local minimizer for $\mathcal{R}_p$ for the indicated values of $p$:

(i) A spherical space form and $p\in[2,\infty)$.

(ii) A hyperbolic manifold and $p\in[\frac{n}{2},\infty)$.

(iii) A product of spherical space forms and  $p\in[2,n]$.

(iv) A product of hyperbolic manifolds and $p\in[\frac{n}{2},n]$.
\end{cor}
\begin{proof} In light of Proposition 3 it is sufficient to prove that $(M,g)$ is $L^{2,2}$-stable.
Define
$$\|h\|^2_1=\|D^*Dh\|^2+\|Dh\|^2+\|h\|^2$$
From the proof of Theorem 1.1 we have that there exists a positive constant $k$ such that $H(h,h)\geq k\|h\|^2_1$ for all $h\in \mathcal{W}.$ When $(M,g)$ has unit volume one can easily check that $\tilde{H}(h,h)=H(h,h)$. Hence to prove the corollary it is sufficient to prove that $\|.\|_{L^{2,2}}$-norm and $\|.\|_1$-norm are equivalent.

Since $M$ is compact and $D^*D$ is an elliptic operator using elliptic estimate, we have $C>0$ such that
$$\|h\|^2_{L^{2,2}}\leq C[\|D^*Dh\|^2+\|h\|^2]$$
Therefore, $\|h\|^2_{L^{2,2}}\leq C\|h\|_1^2$. Since at every point $|D^2h|>|D^*Dh|$ we have $\|h\|^2_1\leq \|h\|^2_{L^{2,2}}$. Hence, the proof follows.
\end{proof}
As a consequence we have the following.
\begin{cor} Let $(M,g)$ be a spherical space form or  product of spherical space forms. There exists a neighborhood $\mathcal{U}$ of $g$ in $\mathcal{M}$ such that for every $g_0\in\mathcal{U}$,

(i) If $\mathcal{R}_p(g_0)<\mathcal{R}_p(g)$ for any $p>\frac{n}{2}$ then $V(g_0)>V(g)$.

(ii) If $\mathcal{R}_p(g_0)<\mathcal{R}_p(g)$ for any $p\in[2, \frac{n}{2})$, then $V(g_0)<V(g)$.

(iii) If $\mathcal{R}_p(g_0)\geq\mathcal{R}_p(g)$ for any $p\in [2,\infty)$ and $V(g_0)=V(g)$, then $g_0$ is isometric to $g$.
\end{cor}
\begin{cor} Let $(M,g)$ be a compact hyperbolic manifold or product of compact hyperbolic manifolds. There exists a neighborhood $\mathcal{V}$ of $g$ in $\mathcal{M}$ such that for every $g_1\in \mathcal{V}$,

(i) If $\mathcal{R}_p(g_1)<\mathcal{R}_p(g)$ for any $p\in(\frac{n}{2},n)$ then $V(g_1)>V(g)$.

(ii) If $\mathcal{R}_p(g_1)\geq\mathcal{R}_p(g)$ for any $p\in [\frac{n}{2},n]$ and $V(g_1)=V(g)$, then $g_1$ is isometric to $g$.
\end{cor}
{\bf Remark 6.2:}
Consider the Lie group $SU(2)$ with bi-invariant metric $g$ which is isometric to the standard sphere $S^3$. Let $\gt(t)$, $t>0$ denote the volume normalized Berger's collapsing metrics on $SU(2)$. Suppose $\tR(t)$ is the restriction of $\tR$ on $\gt(t)$. Since $\tR(t)\to 0$ as $t\to 0$ and $\tR (t) $ has a minima at $\gt(1)$, $\tR(t)$ has a maxima $\gt(t_o)$ for some $t_o$ in between $0$ and $1$. $\gt (t_o)$ is precisely the critical metric for $\tR$ which is exhibited by F. Lamontagne in [LF1] for $p=2$.


\end{document}